\documentclass{article}
\usepackage{amsfonts}
\usepackage{amssymb}
\usepackage{amsmath}
\usepackage{amsthm}

\newtheorem{theorem}{Theorem}
\newtheorem{proposition}{Proposition}

\newtheorem{lemma}{Lemma}
\newtheorem{remark}{Remark}
\newtheorem{corollary}{Corollary}

\renewcommand{\Re}{\mathop{\mathrm{Re}}}
\renewcommand{\Im}{\mathop{\mathrm{Im}}}

\newcommand{\card}{\mathop{\mathrm{card}}}

\newcommand{\SU}{\mathop{\mathrm{SU}}}

\newcommand{\SO}{\mathop{\mathrm{SO}}}
\newcommand{\OO}{\mathop{\mathrm{O{}}}}
\newcommand{\so}{\mathop{\mathrm{so}}}

\newcommand{\cN}{\mathord{\mathcal{N}}}

\newcommand{\cE}{\mathord{\mathcal{E}}}
\newcommand{\cF}{\mathord{\mathcal{F}}}
\newcommand{\cH}{\mathord{\mathcal{H}}}
\newcommand{\cI}{\mathord{\mathcal{I}}}

\newcommand{\cP}{\mathord{\mathcal{P}}}

\newcommand{\cU}{\mathord{\mathcal{U}}}

\newcommand{\frh}{\mathord{\mathfrak{h}}}

\newcommand{\bfn}{{\mathord{\mathbf{n}}}}

\newcommand{\bbC}{{\mathord{\mathbb{C}}}}

\newcommand{\bbP}{{\mathord{\mathbb{P}}}}
\newcommand{\bbR}{{\mathord{\mathbb{R}}}}
\newcommand{\bbS}{{\mathord{\mathbb{S}}}}

\newcommand{\al}{{\mathord{\alpha}}}
\newcommand{\be}{{\mathord{\beta}}}

\newcommand{\Ga}{{\mathord{\Gamma}}}

\newcommand{\om}{{\mathord{\omega}}}
\newcommand{\si}{{\mathord{\sigma}}}
\newcommand{\ka}{{\mathord{\kappa}}}

\newcommand{\la}{{\mathord{\lambda}}}

\newcommand{\de}{{\mathord{\delta}}}
\newcommand{\De}{{\mathord{\Delta}}}

\newcommand{\ep}{{\mathord{\varepsilon}}}

\newcommand{\ze}{{\mathord{\zeta}}}

\newcommand{\spann}{\mathop{\mathrm{span}}\nolimits}
\newcommand{\codim}{\mathop{\mathrm{codim}}\nolimits}
\newcommand{\inr}{\mathop{\mathrm{inr}}\nolimits}
\newcommand{\scal}[2]{\left<#1,#2\right>}

\let\ov=\overline
\let\td=\tilde

\newcounter{ccc}
\date{}
\title{\bf Some remarks on spherical harmonics}
\author{V.M. Gichev\thanks{Partially supported by RFBR grants
06-08-01403, 06-07-89051 and SB RAS project No. 117}}

\begin{document}
\maketitle
\begin{abstract}
The article contains several observations on spherical harmonics
and their nodal sets: a construction for harmonics with prescribed
zeroes; a natural representation of this type for harmonics on
$\bbS^2$; upper and lower bounds for nodal length and inner radius
(the upper bounds are sharp); the precise upper bound for the
number of common zeroes of two spherical harmonics on $\bbS^2$;
the mean Hausdorff measure on the intersection of $k$ nodal sets
for harmonics of different degrees on $\bbS^m$, where $k\leq m$
(in particular, the mean number of common zeroes of $m$
harmonics).
\end{abstract}

\section*{Introduction}
This article contains several observations on spherical harmonics
and their nodal sets; the emphasis is on the case of $\bbS^2$.

Let $M$ be a compact connected homogeneous Riemannian manifold,
$G$ be a compact Lie group acting on $M$ transitively by
isometries,  and $\cE$ be a $G$-invariant subspace of the (real)
eigenspace for some non-zero eigenvalue of the Laplace--Beltrami
operator. We show that each function in $\cE$ can be realized as
the determinant of a matrix, whose entries are values of the
reproducing kernel for $\cE$. There is a similar well-known
construction for the orthogonal polynomials. However, the method
does not work for an arbitrary finite dimensional $G$-invariant
subspace of $C(M)$ (see Remark~\ref{extpo}). There is a natural
unique up to scaling factors realization of this type for
spherical harmonics on $\bbS^2$. It can be obtained by
complexification and restriction to the null-cone $x^2+y^2+z^2=0$
in $\bbC^3$. There is a two-sheeted equivariant covering of this
cone by $\bbC^2$, which identifies the space $\cH_n$ of harmonic
homogeneous complex-valued polynomials of degree $n$ on $\bbR^3$
with the space $\cP^2_{2n}$ of homogeneous holomorphic polynomials
on $\bbC^2$ of degree $2n$. \footnote{In 1876, Sylvester used an
equivalent construction to refine Maxwell's method for
representation of spherical harmonics. According to it, one has to
differentiate the function $1/r$, where $r$ is the distance to
origin, in suitable directions in $\bbR^3$ to get a real harmonic.
The directions are uniquely defined; the corresponding points in
$\bbS^2$ are called {poles} (see \cite[Ch. 9]{Ma} or
\cite[11.5.2]{BE}; \cite[Ch.~7, section~5]{CH} and \cite[Appendix
A]{Ar} contain extended expositions and further information).}

The set of all zeroes of a real spherical harmonic $u$ is called a
{\it nodal set}. We say that $u$ and its nodal set $N_u$ are {\it
regular} if zero is not a critical value of $u$. Then each
component of $N_u$ is a Jordan contour. According to \cite{Gi}, a
pair of the nodal sets $N_u,N_v$, where $u,v\in\cH_n^\bbR$ and
$n>0$, have a non-void intersection; moreover, if $u$ is regular,
then each component of $N_u$ contains at least two points of
$N_v$. The set $N_u\cap N_v$ may be infinite but the family of
such pairs $(u,v)$ is closed and nowhere dense in
$\cH_n^\bbR\times\cH_n^\bbR$. If $N_u\cap N_v$ is finite, then
$\card N_u\cap N_v\leq2n^2$.  The estimate follows from the Bezout
theorem and is precise. This gives an upper bound for the number
of critical points of a generic spherical harmonic, which probably
is not sharp. The configuration of  critical points is always
degenerate in some sense (see Remark~\ref{degen}). The problem of
finding lower bounds seems to be more difficult. According to
partial results and computer experiments, $2n$ may be the sharp
lower bound.

The investigation of metric and topological properties of the
nodal sets has a long and rich history; we only give a few remarks
on the subject of this paper. Let $\De$ be the Laplace--Beltrami
operator and $\la$ be an eigenvalue of $-\De$.

In 1978,  Br\"uning (\cite{Br}) found the lower bound
$c\sqrt{\la}\,$ for the length of a nodal set on a Riemann
surface. Yau conjectured (\cite[Problem 74]{Ya}) that the
Hausdorff measure of the nodal set of a $\la$-eigenfunction on a
compact Riemannian manifold admits upper and lower bounds of the
type $c\sqrt{\la}$. This conjecture was proved by Donnelly and
Fefferman for real analytic manifolds in \cite{DF}. In
(\cite{Sa}), Savo proved that
$\frac{1}{11}\mathop{\mathrm{Area}}(M)\sqrt{\la}$ is the lower
bound for the length of a nodal set in a surface $M$ for all
sufficiently large $\la$ in any surface and for all $\la$ if the
curvature is nonnegative. The upper and lower estimates of the
inner radius were found by Mangoubi (\cite{Man1}, \cite{Man2}); in
the case of surfaces, they are of order $\la^{-\frac12}$
(\cite{Man1}).

One can find the 1-dimensional Hausdorff measure of a set in
$\bbS^2$ integrating over $\SO(3)$ the counting function for the
number of its common points with translates of a suitable subset
of $\bbS^2$ (see Theorem~\ref{feder}). Using estimates of the
number of common zeroes, we give upper and lower bounds for the
length of a nodal set and for the inner radius of a nodal domain
in $\bbS^2$. The upper bounds are precise.

Let $\cH^{m+1}_n$ be the space of all real spherical harmonics of
degree $n$ on the unit sphere $\bbS^m$ in $\bbR^{m+1}$.
Corresponding to a point of $\bbS^m$ the evaluation functional at
it on $\cH^{m+1}_n$, we get an equivariant immersion of $\bbS^m$
to the unit sphere in $\cH^{m+1}_n$, which is locally a metric
homothety with the coefficient $\sqrt{\frac{\la_n}{m}}$, where
$\la_n=n(n+m-1)$ is the eigenvalue of $-\De$ in $\cH^{m+1}_n$.
This makes it possible to calculate the mean Hausdorff measure of
the intersection of $k$ harmonics of degrees $n_1,\dots,n_k$: it
is equal to $c\sqrt{\la_{n_1}\dots \la_{n_k}}$, where $c$ depends
only on $m$ and $k$ and $k\leq m$ (Theorem~\ref{meanh}). In
particular, if $k=m$, then we get the mean number of common zeroes
of $m$ harmonics: it is equal to
$2m^{-\frac{m}{2}}\sqrt{\la_{n_1}\dots \la_{n_m}}$; if $m=2$, then
$\sqrt{\la_{n_1}\la_{n_2}}$. In article \cite{DF}, Donelly and
Fefferman wrote: ``A main theme of this paper is that a solution
of $\De F=-\la F$, on a real analytic manifold, behaves like a
polynomial of degree $c\sqrt{\la}$\,''. Following this idea,
L.\,Polterovich conjectured that the mean number of common zeroes
is subject to the Bezout theorem, i.e., that it is as above. Thus,
the result in the case $k=m$ confirms this conjecture up to
multiplication by a constant, and may be treated as ``the Bezout
theorem in the mean'' for the spherical harmonics. For $k=1$, the
mean Hausdorff measure, by different but similar methods, was
found by Berard in \cite{Be} and Neuheisel in \cite{Ne}. The case
of a flat torus was investigated by Rudnick and Wigman
(\cite{RW}).

\section{Construction of eigenfunctions which vanish on prescribed finite sets}

In this section, $M$ is a compact connected oriented homogeneous
Riemannian manifold of a compact Lie group $G$ acting by
isometries on $M$, $\De$ is the Laplace--Beltrami operator on $M$,
\begin{eqnarray}
\la>0\label{lanze}
\end{eqnarray}
is an eigenvalue of $-\De$, $\cE_\la$ is the corresponding real
eigenspace (i.e., $\cE_\la$ consists of real valued
eigenfunctions), and $\cE$ is its $G$-invariant linear subspace.
Thus, $\cE$ is a finite sum of $G$-invariant irreducible subspaces
of $C^\infty(M)$. The invariant measure with the total mass 1 on
$M$ is denoted by $\si$, $L^2(M)=L^2(M,\si)$. For any $a\in M$,
there exists the unique $\phi_a\in\cE$ that realizes the
evaluation functional at $a$:
\begin{eqnarray*}
\scal{u}{\phi_a}=u(a)
\end{eqnarray*}
for all $u\in\cE$. Set
\begin{eqnarray*}
\phi(a,b)=\phi_a(b),\quad a,b\in M.
\end{eqnarray*}
It follows that
\begin{eqnarray}
&\phi(a,b)=\phi_a(b)=\scal{\phi_a}{\phi_b}=\scal{\phi_b}{\phi_a}=
\phi_b(a)=\phi(b,a),\label{reprm}\\
&u(x)=\scal{u}{\phi_x}=\int\phi(x,y)u(y)\,d\si(y)\quad\mbox{for
all}\ u\in \cE,\label{reprod}\\
&x\in N_u\quad\Longleftrightarrow \quad\phi_x\perp
u,\label{orthm}\\ &\phi_x\neq0\quad\mbox{for all}\quad x\in
M.\label{nonze}
\end{eqnarray}
The latter holds due to the homogeneity of $M$. According to
(\ref{reprod}), $\phi(x,y)$ is the {\it reproducing kernel for
$\cE$} (i.e., the mapping $u(x)\to\int\phi(x,y)u(y)\,d\si(y)$ is
the orthogonal projection onto $\cE$ in $L^2(M)$).

Let $a_1,\dots,a_k,x,y\in M$. Set $a=(a_1,\dots,a_k)\in M^k$ and
let $a$ also denote the corresponding $k$-subset of $M$:
$a=\{a_1,\dots,a_k\}$. Set
\begin{eqnarray}\label{defph}
\Phi_k^a(x,y)=\Phi_{k,y}^a(x)=\det\left(\begin{array}{cccc}
\phi(a_1,a_1)&\dots&\phi(a_1,a_k)&\phi(a_1,y)\\
\vdots&\ddots&\vdots&\vdots\\
\phi(a_k,a_1)&\dots&\phi(a_k,a_k)&\phi(a_k,y)\\
\phi(x,a_1)&\dots&\phi(x,a_k)&\phi(x,y)
\end{array}\right).
\end{eqnarray}
Obviously, $\Phi_k^a(x,y)=\Phi_k^a(y,x)$. Let us fix $y$ and set
$v=\Phi_{k,y}^a$. Then, by (\ref{defph}), $v\in\cE$ and
\begin{eqnarray}\label{akinn}
a_1,\dots a_k\in N_v.
\end{eqnarray}
We say that $a_1,\dots a_k$ are {\it independent} if the vectors
$\phi_{a_1},\dots,\phi_{a_k}\in\cE$ are linearly independent. For
a subset $X\subseteq M$, put
\begin{eqnarray}\label{defnx}
\cN_X=\spann\{\phi_x:\,x\in X\}.
\end{eqnarray}
If 
$X=N_u$, where $u\in\cE$,
then we abbreviate the notation: 
$\cN_{N_u}=\cN_u$. 
Set $$n=\dim\cE-1.$$ It follows from (\ref{lanze}) that $n\geq1$
(note that $\cE$ is real and $G$-invariant).
\begin{lemma}\label{linin}
Let $a\in M^k$, where $k\leq n$. Then $a_1,\dots a_k$ are
independent if and only if $\Phi_{k,y}^a\neq0$ for some $y\in M$.
\end{lemma}
\begin{proof}
It follows from (\ref{orthm}) that $\cE=\cN_M$; since $k\leq n$,
$\cN_a\neq\cE$. Therefore, if $a_1,\dots a_k$ are independent,
then we get an independent set adding $y$ to $a$, for some $y\in
M$. Then $\Phi_{k,y}^a\neq0$ since $\Phi_{k,y}^a(y)>0$ (by
(\ref{reprm}) and (\ref{defph}), $\Phi_{k,y}^a(y)$ is the
determinant of the Gram matrix for the vectors
$\phi_{a_1},\dots,\phi_{a_k},\phi_y$). Clearly, $\Phi_{k,y}^a=0$
for all $y\in M$ if $a_1,\dots a_k$ are  dependent.
\end{proof}
The following proposition implies that each function in $\cE$ can
be realized in the form (\ref{defph}).
\begin{proposition}\label{bigno}
For any $u\in \cE$, $\cN_u=u^\bot\cap\cE$.
\end{proposition}
\begin{lemma}\label{key}
If $u,v\in\cE$ and $N_v\supseteq N_u$, then $v=cu$ for some
$c\in\bbR$.
\end{lemma}
\begin{proof}
This immediately follows from the inclusion $N_v\supseteq N_u$ and
Lemma~1 of \cite{Gi}, which states that $v=cu$ for some
$c\in{\mathbb R}$ if there exist nodal domains  $U,V$ for $u,v$,
respectively, such that $V\subseteq U$.
\end{proof}
Here is a sketch of the proof of the mentioned lemma; it is based
on the same idea as Courant's Nodal Domain Theorem. Since $u$ does
not change its sign in $U$, $\la$ is the first Dirichlet
eigenvalue for $U$. Hence, it has multiplicity 1 and
$D(w)\geq\la\|w\|_{L^2(U)}$ for all $w\in C^2(M)$ that vanish on
$\partial U$, where $D$ is the Dirichlet form on $U$. Moreover,
the equality holds if and only if $w=cu$ for some $c\in\bbR$. On
the other hand, if $w$ vanishes outside $V$ and coincides with $v$
in $V$, then the equality is fulfilled.

\begin{proof}[Proof of Proposition~\ref{bigno}]
If $v\in\cE$ and $v\perp\cN_u$, then $N_v\supseteq N_u$ by
(\ref{orthm}). Thus,  $v\in\bbR u$ by Lemma~\ref{key}. Therefore,
$\cN_u\supseteq u^\bot\cap\cE$. The reverse inclusion is evident.
\end{proof}

Let $\Phi:\,M^{n+1}\to\cE\,$ be the mapping $(a,y)\to\Phi_{n,y}^a$
and set $\,\cU=\Phi(M^{n+1})$.
\begin{theorem}\label{repm}
\begin{itemize}
\item[\rm({\romannumeral1})] Let $u\in\cE$, $u\neq0$. For
$(a,y)\in N_u^{n}\times M$,
\begin{eqnarray}\label{deter}
\Phi(a,y)=c(a,y)u,
\end{eqnarray}
where $c$ is a continuous nontrivial function on $N_u^{n}\times
M$. \item[\rm({\romannumeral2})] $\,\cU$ is a compact symmetric
neighbourhood of zero in $\cE$. \item[\rm({\romannumeral3})] For
every $a\in M^{n}$, there exists a nontrivial nodal set which
contains $a$; for a generic $a$, this set is unique.
\end{itemize}
\end{theorem}
\begin{proof}
Let $a\in N_u^{n}$. If $a_1,\dots,a_n$ are independent, then
$\codim\cN_a=1$; since $u\perp\cN_u$ by (\ref{orthm}), we get
(\ref{deter}), where $c(a,y)\neq0$ for some $y\in M$ by
Lemma~\ref{linin}. If $a_1,\dots,a_n$ are dependent, then
$\Phi(a,y)=0$ for all $y\in M$ by the same lemma. The function $c$
is continuous by (\ref{defph}); it is nonzero since the set $N_u$
contains independent points $a_1,\dots,a_n$ by
Proposition~\ref{bigno}. This proves ({\romannumeral1}).

According to (\ref{defph}), $\Phi$ is continuous. Hence, $\cU$ is
compact. Since $M$ is connected, for any $u\in\cU$, we may get the
segment $[0,u]$ moving $y$; hence, $\cU$ is starlike. Since
transposition of every two points in $a$ changes the sign of
$c(a,y)$, $\cU$ is symmetric if $n>1$; for $n=1$, $\cU$ is a disc
in $\cE$ because it is $G$-invariant and starlike. Thus, $\cU$ is
compact, symmetric, starlike, and $\cup_{t>0}\,t\cU=\cE$. Hence
$\cU$ is a neighbourhood of zero, i.e., ({\romannumeral2}) is
true.

Let $a\in M^{n}$ and $a'\subseteq a$ be a maximal independent
subset of $a$. Then $\Phi_{k,y}^{a'}\neq0$ for some $y\in M$ by
Lemma~\ref{linin}, where $k=\card a'$. Set $v=\Phi_{k,y}^{a'}$.
According to (\ref{akinn}), $a'\subset N_v$. By (\ref{orthm}),
$N_v$ contains any point $x\in M$ such that $\phi_x\in\cN_{a'}$.
Hence $N_v$ includes $a$.  The set $N_v$ is unique if
$a_1,\dots,a_n$ are independent because $\codim\cN_v=1$ in this
case. Since $M$ is homogeneous and $\cE$ is finite dimensional,
the functions $\phi_x$, $x\in M$, are real analytic. Therefore,
either $\Phi^a_{n,y}=0$ for all $(a,y)\in M^{n+1}$ or
$\Phi^a_{n,y}\neq0$ for generic $(a,y)$ (note that $M$ is
connected). Finally, $\Phi^a_{n,y}\neq0$ for some $(a,y)\in
M^{n+1}$ since $\cN_M=\cE$ due to (\ref{orthm}) and (\ref{nonze}).
\end{proof}
A closed subset $X\subseteq M$ is called an {\it interpolation set
for a function space $\cF\subseteq C(M)$} if $\cF|_X=C(X)$.
\begin{corollary}\label{inter}
Let $k\leq\dim\cE$. For generic $a_1,\dots,a_k\in M$,
$a=\{a_1,\dots,a_k\}$ is an interpolation set for $\cE$. \qed
\end{corollary}
\begin{remark}\rm
The function $c$ may vanish on some components of the set
$N_u^{n}\times M$. For example, let $M$ be the unit sphere
$\bbS^2\subset\bbR^3$ and $\cE$ be the restriction to it of the
space of harmonic homogeneous polynomials of degree $k$; then
$\dim\cE=2k+1$, $n=2k$. If $k>1$, then any big circle $\bbS^1$ in
$\bbS^2$ is contained in several nodal sets (for example, nodal
sets of the functions $x_1f(x_2,x_3)$, where $f$ is harmonic,
contain the big circle $\{x_1=0\}\cap\bbS^2$); moreover, if $k$ is
odd, then $\bbS^1$ may be a component of $N_u$. Hence,
$\codim\cN_{\bbS^1}>1$ and $\Phi(a,y)=0$ for all
$(a,y)\in\left(\bbS^1\right)^{n}\times \bbS^2$.
\end{remark}
\begin{remark}\label{extpo}\rm
Theorem~\ref{repm} fails for a generic finite dimensional
$G$-invariant subspace $\cE\subseteq C(M)$. Indeed, if $\dim\cE>1$
and $\cE$ contains  constant functions, then it includes an open
subset consisting of functions without zeroes, which evidently
cannot be realized in the form (\ref{defph}). Furthermore, it
follows from the theorem that the products
$\phi_{a_1}\wedge\dots\wedge\phi_{a_{n}}$ fill a neighbourhood of
zero in the $n$th exterior power of $\cE$, which may be identified
with $\cE$. This property evidently imply the interpolation
property of Corollary~\ref{inter} but the converse is not true; an
example is the space of all homogeneous polynomials of degree
$m>1$ on $\bbR^3$, restricted to $\bbS^2$ (or the space of all
polynomials of degree less than $n$ on $[0,1]\subset\bbR$, where
$n>2$).
\end{remark}

\section{Spherical harmonics on $\bbS^2$}

Let $\cP^m_n$ denote the space of al homogeneous holomorphic
polynomials of degree $n$ on $\bbC^m$ or/and the space of all
complex valued homogeneous polynomials of degree $n$ on $\bbR^m$;
clearly, there is one-to-one correspondence between these spaces.
Its subspace of polynomials which are harmonic on $\bbR^m$ is
denoted by $\cH_n^m$; we omit the index $m$ in $\cH_n^m$ if $m=3$.
Then $\dim\cH_n=2n+1$. The polynomials in $\cH_n^m$, as well as
their traces on the unit sphere $\bbS^{m-1}\subset\bbR^m$, are
called {\it spherical harmonics}. They are eigenfunctions of the
Laplace--Beltrami operator; if $m=3$, then the eigenvalue is
$-n(n+1)$. For a proof of these facts, see, for example,
\cite{SW}. We say that $u\in\cP^m_n$ is {\it real} if it takes
real values on $\bbR^m$.

The standard inner product in $\bbR^m$ and its bilinear extension
to $\bbC^m$ will be denoted by $\scal{\ \,}{\ }$,
$$r(v)=|v|=\sqrt{\scal{v}{v}},\quad v\in\bbR^m,$$
$r^2$ is a holomorphic quadratic form on $\bbC^m$. For
$a\in\bbC^m$, set
$$l_a(v)=\scal{a}{v}.$$

The functions $\Phi_k^a(x,y)$ admit holomorphic extensions on all
variables (except for $k$). If $M=\bbS^2\subset\bbR^3$, then the
extension to $\bbC^3$ and subsequent restriction to the null-cone
\begin{eqnarray*}
S_0=\{z\in\bbC^3:\,r^2(z)=0\}
\end{eqnarray*}
makes it possible to construct a kind of a natural representation
in the form (\ref{defph}), which is unique up to multiplication by
a complex number, for any complex valued spherical harmonic. The
projection of $S_0$ to $\bbC\bbP^2$ is Riemann sphere
$\bbC\bbP^1$. The cone $S_0$ admits a natural parametrization:
\begin{eqnarray}\label{param}
\ka(\ze_1,\ze_2)=(z_1,z_2,z_3)=(2\ze_1\ze_2,\ze_1^2-\ze_2^2,i(\ze_1^2+\ze_2^2)),
\quad \ze_1,\ze_2\in\bbC.
\end{eqnarray}
\begin{lemma}\label{restr}
The mapping $R:\,\cH_n\to\cP_{2n}^2$ defined by
\begin{eqnarray*}
Rp=p\circ\ka
\end{eqnarray*}
is one-to-one and intertwines the natural representations of
$\SO(3)$ and $\SU(2)$ in $\cH_n$ and $\cP_{2n}^2$, respectively.
\end{lemma}
\begin{proof}
Clearly,  $p\circ\ka$ is a homogeneous polynomial on $\bbC^2$ of
degree $2n$ for any $p\in\cP_n^3$. Further, $\ka$ is equivariant
with respect to the natural actions of $\SU(2)$ in $\bbC^2$ and
$\SO(3)$ in $\bbC^3$: an easy calculation with  (\ref{param})
shows that the change of variables $\ze_1\to a\ze_1+b\ze_2$,
$\ze_2\to-\ov b\ze_1+\ov a\ze_2$, where $|a|^2+|b|^2=1$, induces a
linear transformation in $\bbC^3$ which preserves $r^2$ and leaves
$\bbR^3$ invariant (in other words, the transformation of
$\cP^2_2$, induced by this change of variables, in the base
$2\ze_1\ze_2$, $\ze_1^2-\ze_2^2$, $i(\ze_1^2+\ze_2^2)$ corresponds
to a matrix in $\SO(3)$). Hence $R$ is an intertwining operator.
It is well known that
\begin{eqnarray*}
\cP_n^3=\cH_n\oplus r^2\cP_{n-2}^3
\end{eqnarray*}
(see, for example, \cite{SW}). Since $R\neq0$ and $Rr^2=0$, we get
$R\cH_n\neq0$. It remains to note that the natural representations
of these groups in $\cH_n$, $\cP^2_n$ are irreducible.
\end{proof}
\begin{corollary}\label{propo}
For any $p\in\cH_n\setminus\{0\}$, the set $p^{-1}(0)\cap S_0$ is
the union of $2n$ complex lines; some of them may coincide. If the
lines are distinct, $q\in\cH_n$, and $p^{-1}(0)\cap
S_0=q^{-1}(0)\cap S_0$, then $q=cp$ for some $c\in\bbC$.
\end{corollary}
\begin{proof}
Clearly, $\ka$ maps lines onto lines and induces an embedding of
$\bbC\bbP^1$ into $\bbC\bbP^2$.
\end{proof}

The functions $\phi_a$ of the previous section can be written
explicitly:
\begin{eqnarray*}
\phi_a(x)=c_nP_n(\scal{a}{x}),\quad\mbox{where}\quad a,x\in
\bbS^2,
\end{eqnarray*}
$c_n$ is a normalizing constant, and $P_n$ is the $n$th Legendre
polynomial:
$P_n(t)=\frac1{2^{n}n!}\frac{d^{n}}{dt^{n}}(t^2-1)^{2n}$. There is
the unique extension of
$$\phi(a,x)=\phi_a(x)$$
into $\bbR^3$ which is homogeneous of degree $n$ and harmonic on
both variables (it is also symmetric and extends into $\bbC^3$
holomorphically). For example, if $n=3$, then $2P_3(t)=5t^3-3t$
and $\phi(a,x)$ is proportional to
$$5\scal{a}{x}^3-3\scal{a}{a}\scal{a}{x}\scal{x}{x}$$
(if $a=(1,0,0)$, then to $2x_1^3-3x_1x_2^2-3x_1x_3^2$). Of course,
the representation of $p\in\cH_n$ in the form (\ref{defph}) holds
for $M=\bbS^2$ but there is a more natural version in this case.
For $\ze=(\ze_1,\ze_2)\in\bbC^2$, set
\begin{eqnarray*}
j\ze=(-\ze_2,\ze_1).
\end{eqnarray*}
\begin{theorem}
Let $p\in\cH_n$. Suppose that $p^{-1}(0)\cap S_0$ is the union of
distinct lines $\bbC a_1,\dots,\bbC a_{2n}$. Then there exists a
constant $c\neq0$ such that
\begin{eqnarray}\label{reps3}
p(x)p(y)=c\det\left(\begin{array}{cccc}
\scal{a_1}{a_1}^{n}&\dots&\scal{a_1}{a_{2n}}^{n}&\scal{a_1}{y}^{n}\\
\vdots&\ddots&\vdots&\vdots\\
\scal{a_{2n}}{a_1}^{n}&\dots&\scal{a_{2n}}{a_{2n}}^{n}&\scal{a_{2n}}{y}^{n}\\
\scal{x}{a_1}^{n}&\dots&\scal{x}{a_{2n}}^{n}&\scal{x}{y}^{n}
\end{array}\right)
\end{eqnarray}
for all $y\in S_0$, $x\in\bbC^3$. Moreover, replacing
$\scal{x}{y}^{n}$ with $\phi(x,y)$ in the matrix, we get such a
representation of $p(x)p(y)$ for all $x,y\in\bbC^3$ (with another
$c$ in general).
\end{theorem}
\begin{proof}
A calculation shows that $\scal{a}{x}^{n}$ is harmonic on $x$ for
all $n$ if $a\in S_0$. Hence, the function
$\Phi^a_y(x)=\Phi^a(x,y)$ in the right-hand side belongs to
$\cH_n$ for each $y\in S_0$. Clearly, $\Phi^a_y(a_k)=0$ for all
$k=1,\dots,2n$. By Corollary~\ref{propo}, $\Phi^a_y$ is
proportional to $p$. Since $\Phi^a(x,y)=\Phi^a(y,x)$, we get
(\ref{reps3}) if the right-hand side is nontrivial. Thus, we have
to prove that $c\neq0$. Let $x\in S_0$. There exist
$\al_1,\dots,\al_{2n},\xi,\eta\in\bbC^2$ such that
$a_k=\ka(\al_k)$ for all $k$, $x=\ka(\xi)$, and $y=\ka(\eta)$. By
a straightforward calculation, for any $a,b\in\bbC^2$ we get
\begin{eqnarray}\label{toc2}
\scal{\ka(a)}{\ka(b)}=
-2\scal{a}{jb}^2.
\end{eqnarray}
Hence, the right-hand side of (\ref{reps3}) is equal to
\begin{eqnarray}\label{reps2}
-2^{(2n+1)n}c\det\left(\begin{array}{cccc}
\scal{\al_1}{{j\al_1}}^{2n}&\dots&\scal{\al_1}{{j\al_{2n}}}^{2n}&\scal{\al_1}{{j\eta}}^{2n}\\
\vdots&\ddots&\vdots&\vdots\\
\scal{\al_{2n}}{{j\al_1}}^{2n}&\dots&\scal{\al_{2n}}{{j\al_{2n}}}^{2n}&\scal{\al_{2n}}{{j\eta}}^{2n}\\
\scal{\xi}{{j\al_1}}^{2n}&\dots&\scal{\xi}{{j\al_{2n}}}^{2n}&\scal{\xi}{{j\eta}}^{2n}
\end{array}\right).
\end{eqnarray}
The determinant can be calculated explicitly. More generally, if
$C=(c_{rs})_{r,s=1}^{k+1}$, where $c_{rs}=\scal{a_r}{b_s}^{k}$,
$a_r,b_s\in\bbC^2$, then
\begin{eqnarray}\label{detcc}
\det C=\prod_{r=1}^{k}{{k}\choose r}\prod_{s<r}\scal{a_r}{ja_s}
\prod_{s<r}\scal{b_r}{jb_s}
\end{eqnarray}
Let $a_r=(a_{r,1},a_{r,2})$, $b_s=(b_{s,1},b_{s,2})$. If all the
entries are nonzero, then
\begin{eqnarray*}
c_{rs}=\sum_{t=0}^k {{k}\choose r}
(a_{r,1}b_{s,1})^t(a_{r,2}b_{s,2})^{k-t}=a_{r,2}^kb_{s,1}^k
\sum_{t=0}^k {{k}\choose
r}\left(\frac{a_{r,1}}{a_{r,2}}\right)^t\left(\frac{b_{s,2}}{b_{s,1}}\right)^{k-t}
\end{eqnarray*}
We may factor out rows and columns of $C$. Then we get a matrix
$\td C$, which admits the decomposition $\td C=AB$, where
\begin{eqnarray*}
&A=\big({{k}\choose r}\al^t_r\big)_{r,t=0}^k,\quad
B=\big(\be^{k-t}_s\big)_{t,s=0}^k,\quad
\al_r=\frac{a_{r,1}}{a_{r,2}},\quad \be_s=\frac{b_{s,2}}{b_{s,1}}.
\end{eqnarray*}
Thus, the computation of $\det C$ is reduced to the Vandermonde
determinant. The straightforward calculation proves (\ref{detcc});
obviously, the assumption that the entries are nonzero is not
essential. Due to (\ref{detcc}) and (\ref{toc2}), this implies
that the determinant in (\ref{reps2}) is not zero if the lines
$\bbC\xi,\bbC\eta,\bbC a_1,\dots,\bbC a_{2n}$ are distinct (if
$a\in S_0$, then the plane $\scal{z}{a}=0$ intersects $S_0$ in the
line $\bbC a$). Hence, $c\neq0$.

It follows from the definition of $P_n$ and $\phi$ that
\begin{eqnarray}\label{vyrph}
\phi(x,y)=s_n\scal{x}{y}^n+r^2(x)r^2(y)h(x,y),
\end{eqnarray}
there $s_n>0$ is constant and $h$ is a polynomial. Therefore, we
can  get a function $f\neq0$ on $\bbC^3$, which coincides with
$p(x)$ on $S_0$ up to a constant factor, replacing $\scal{x}{y}^n$
with $\phi(x,y)$ in (\ref{reps3}) and fixing generic $y\in\bbC^3$.
By Corollary~\ref{propo}, the same is true on $\bbC^3$ since
$f\in\cH_n$ according to (\ref{reps3}) (all functions in the last
row are harmonic on $x$). Since $\phi(x,y)=\phi(y,x)$, this proves
the second assertion.
\end{proof}
\begin{remark}\rm
The set $p^{-1}(0)\cap S_0$, where $p\in\cH_n$, is also
distinguished by the orthogonality condition
\begin{eqnarray*}
\int_{\bbS^2}p(x)\scal{x}{w}^n\,d\si(x)=0,
\end{eqnarray*}
where $\si$ is the invariant measure on $\bbS^2$ and $w\in S_0$.
This is a consequence of (\ref{vyrph}) since $\int
p(x)\phi(x,y)\,d\si(x)=p(y)$ for all $y\in \bbS^2$, hence for all
$y\in\bbR^3$ ($p(y)$ and $\phi_x(y)$ are homogeneous of degree
$n$), moreover, for all $y\in\bbC^3$ (both sides are holomorphic
on $y$). In particular, this is true for $y\in S_0$ but
$\phi(x,y)=s_n\scal{x}{y}^n$ in this case.

If $p^{-1}(0)\cap S_0$ is the union of distinct lines $\bbC a_k$,
$k=1,\dots,2n$, then the functions $\scal{x}{a_k}^n$,
$k=1,\dots,2n$, form a linear base for the space of functions in
$\cH_n$ which are orthogonal to $p$ with respect to the bilinear
form $\int fg\,d\si$. This is a consequence of (\ref{toc2}): it is
easy to check that the functions $\scal{\ze}{b_s}^{k}$ on
$\bbC^2$, where $s=1,\dots,k$, are linearly independent if the
lines $\bbC b_s$ are distinct (the Vandermonde determinant).\qed
\end{remark}
We conclude this section with remarks on number of zeroes in
$\bbS^2$ of functions in $\cH_n$. Let $f\in\cH_n$, $u=\Re f$,
$v=\Im f$. A zero of $f$ is a common zero of $u$ and $v$. The
following proposition, in a slightly more general form, was proved
in \cite{Gi}. We say that $u$ is {\it regular} if zero is not a
critical value for $u$.
\begin{proposition}[\cite{Gi}]\label{fromp}\label{twopo}
Let $n>0$, $u\in\cH_n$. If $u$ is regular, then for any
$v\in\cH_n$ each connected component of $N_u$ contains at least
two points of $N_v$.\qed
\end{proposition}
The assertion follows from the Green formula which implies that
\begin{eqnarray}\label{grefo}
\int_C v\frac{\partial u}{\partial n}\,ds=0,
\end{eqnarray}
where $C$ is a component of $N_u$, which is a Jordan contour, $ds$
is the length measure on $C$, and $\frac{\partial u}{\partial n}$
is the normal derivative; note that $\frac{\partial u}{\partial
n}$ keeps its sign on $C$. For the standard sphere $\bbS^2$,
(\ref{grefo}) follows from the classical Green formula for the
domain $D_\ep=(1-\ep,1+\ep)\times\bbS^2$, where $\ep\in(0,1)$, and
the homogeneous of degree $0$ extensions of $u,v$ into $D_\ep$.

Let $u,v\in\cH_n$ be real and regular. Set
\begin{eqnarray*}
\nu(u,v)=\card N_u\cap N_v.
\end{eqnarray*}
For singular $u,v$, zeroes must be counted with multiplicities; if
$u,v\in\cH_n$, then the multiplicity of a zero can be defined as
the number of smooth nodal lines which meet at it; if $u,v$ have
multiplicities $k,l$ at their common zero, then one have to count
them $kl$ times (the greatest number of common zeroes which appear
under small perturbations). If $u=\phi_a$, where $a\in \bbS^2$,
then $N_u$ is the union of $n$ parallel circles $\scal{x}{a}=t_k$,
$x\in\bbS^2$, where $k=1,\dots,n$ and $t_1,\dots,t_n$ are the
zeroes of $P_n(t)$. Since they are distinct, $P'_n(t_k)\neq0$ for
all $k$. It follows from Proposition~\ref{fromp} that for any real
$v\in\cH_n$
\begin{eqnarray*}
\nu(\phi_a,v)\geq2n,
\end{eqnarray*}
where $a\in\bbS^2$. If $b\in \bbS^2$ is sufficiently close to $a$,
then the equality holds for $v=\phi_b$. In the inequality above,
$\phi_a$ and $n$ may be replaced with any regular $u$ and the
number of components of $N_u$, respectively. The latter can be
less than $n$ (according to \cite{Le}, it can be equal to one or
two if $n$ is odd or even, respectively\footnote{The corresponding
harmonic is a small perturbation of the function
$\Re(x_1+ix_2)^n$.
}). However, computer experiments support the following
{conjecture}:
for all real $u,v\in\cH_n$,
\begin{eqnarray*}
\nu(u,v)\geq2n.
\end{eqnarray*}
The common zeroes must be counted with multiplicities. Otherwise,
there is a simple example of two harmonics which have only two
common zeroes: $\Re(x_1+ix_2)^n$ and $\Im(x_1+ix_2)^n$.

On the other hand, for generic real $u,v\in\cH_n$ there is a
trivial sharp upper bound for $\nu(u,v)$. We prove a version that
is stronger a bit.
\begin{proposition}\label{verkh}
Let $u,v\in\cH_n$ be real. If $\nu(u,v)$ is finite, then
\begin{eqnarray}\label{upper}
\nu(u,v)\leq2n^2.
\end{eqnarray}
\end{proposition}
By the Bezout theorem, if $u,v\in\cP^3_n$ have no proper common
divisor, then the set $\{z\in\bbC^3:\,u(z)=v(z)=0\}$ is the union
of $n^2$ (with multiplicities) complex lines. Then
$\nu(u,v)\leq2n^2$ since each line has at most two common points
with $\bbS^2$. The proposition is not an immediate consequence of
this fact since $u,v$ may have a nontrivial common divisor which
has a finite number of zeroes in $\bbS^2$. This cannot happen for
$u,v\in\cH_n$ by the following lemma.
\begin{lemma}\label{zerpr}
Let $u\in\cH_n$ be real, $x\in\bbS^2$, and $u(x)=0$. Suppose that
$u=vw$, where $v\in\cP^3_m$, $w\in\cP^3_{n-m}$ are real. If
$w(y)\neq0$ for all $y\in\bbS^2\setminus\{x\}$ that are
sufficiently close to $x$, then $w(x)\neq0$.
\end{lemma}
\begin{proof}
We may assume $x=(0,0,1)$. If $u$ has a zero of multiplicity $k$
at $x$, then
\begin{eqnarray*}
u(x_1,x_2,x_3)=p_k(x_1,x_2)x_3^{n-k}+
p_{k+1}(x_1,x_2)x_3^{n-k-1}+\dots+p_{n}(x_1,x_2),
\end{eqnarray*}
where $p_j\in\cP_j^2$, $p_k\neq0$. Since $\Delta u=0$, we have
$\De p_k=0$. Hence,
\begin{eqnarray*}
p_k(x_1,x_2)=\Re(\la(x_1+ix_2)^k)
\end{eqnarray*}
for some $\la\in\bbC\setminus\{0\}$. Therefore, $p_k$ is the
product of $k$ distinct linear forms. Let
\begin{eqnarray*}
&w=q_l(x_1,x_2)x_3^{n-m-l}+
q_{l+1}(x_1,x_2)x_3^{n-m-l-1}+\dots+q_{n-m}(x_1,x_2),\\
&v=r_s(x_1,x_2)x_3^{m-s}+
r_{s+1}(x_1,x_2)x_3^{m-s-1}+\dots+r_{m}(x_1,x_2),
\end{eqnarray*}
where $q_j,r_j\in\cP^2_j$ and $q_l,r_s\neq0$.  Since
$p_k=q_lr_{s}$, we have $k=l+s$; moreover, either $q_l$ is
constant or it is the product of distinct linear forms. The latter
implies that it change its sign near $x$; then the same is true
for $w$, contradictory to the assumption. Hence $l=0$. Thus,
$q_l\neq0$ implies $w(x)=q_l(x)\neq0$.
\end{proof}
\begin{proof}[Proof of Proposition~\ref{verkh}]
Let $u,v\in\cH_n$ be real and $w$ be their greatest common
divisor. Clearly, $w$ is real. Since $N_u\cap N_v$ is finite,
zeroes of $w$ in $\bbS^2$ must be isolated; by Lemma~\ref{zerpr},
$w$ has no zero in $\bbS^2$. Applying the Bezout theorem to $u/w$
and $v/w$, we get the assertion.
\end{proof}
The equality in (\ref{upper}) holds, for example, for the
following pairs and for their small perturbations:
\begin{eqnarray}
&u=\phi_a,\quad v=\Re(x_2+ix_3)^n,\quad\mbox{where}\quad a={(1,0,0)};\label{maxcr}\\
&u=\Re(ix_2+x_3)^n,\quad v=\Re(x_1+ix_2)^n.\nonumber
\end{eqnarray}
\begin{corollary}\label{cries}
If the number of critical points for real $u\in\cH_n$ is finite,
then it does not exceed $2n^2$; in particular, this is true for a
generic real $u\in\cH_n$.
\end{corollary}
\begin{proof}
If $x$ is a critical point of $u$, then $\xi u(x)=0$ for any
vector field $\xi\in\so(3)$. It is possible to choose two fields
$\xi,\eta\in\so(3)$ which do not annihilate $u$ and are
independent at all critical points; then the critical points of
$u$ are precisely the common zeroes of $\xi u,\eta u\in\cH_n$.
\end{proof}
\begin{remark}\rm
This bound is not sharp. At least, for $n=1,2$ the number of
critical points is equal to $2(n^2-n+1)$, if it is finite. Let
$u,v$ be as in (\ref{maxcr}). Then $u+\ep v$, where $\ep$ is
small, has $2(n^2-n+1)$ critical points. I know no example of a
spherical harmonic with a greater (finite) number of critical
points.
\end{remark}
\begin{remark}\rm\label{degen}
The consideration above proves a bit more than
Corollary~\ref{cries} says. A nontrivial orbit of $u$ under
$\SO(3)$ is either 3-dimensional or 2-dimensional, and the latter
holds if and only if $u=c\phi_a$ for some constant $c$ and $a\in
\bbS^2$. In the first case, the set $C$ of critical points of $u$
is precisely the set of common zeroes of three linearly
independent spherical harmonics (a base for the tangent space to
the orbit of $u$). Hence, $\codim\cN_C\geq3$. Note that generic
three harmonics have no common zero. Thus, {the configuration of
critical points is always degenerate}. The problem of estimation
of the number of critical points, components of nodal sets, nodal
domains, etc., for spherical harmonics on $\bbS^2$ was stated in
\cite{AV}.
\end{remark}
\begin{proposition}\label{finde}
The set $\cI$ of functions  $f=u+iv\in\cH_n$ such that
$\nu(u,v)=\infty$ is closed and nowhere dense in $\cH_n$.
\end{proposition}
\begin{proof}
If $N_u\cap N_v$ is infinite, then it contains a Jordan arc which
extends to a contour since $u$ and $v$ are real analytic (by
\cite{Ch}, a nodal set, outside of its finite subset, is the
finite union of smooth arcs). This contour cannot be included into
a disc $D$ which is contained in some of nodal domains: otherwise,
its first Dirichlet eigenvalue would be greater than $n(n+1)$.
Therefore, diameter of the contour is bounded from below. This
implies that $\cI$ is closed. If $f\in\cI$, then $u$ and $v$ have
a nontrivial common divisor due to the Bezout theorem; hence,
$\cI$ is nowhere dense.
\end{proof}
In examples known to me, if $f\in\cI$, then $N_u\cap N_v$ is the
union of circles.

\section{Estimates of nodal length and inner radius}
Let $M$ be a $C^\infty$-smooth compact connected Riemannian
manifold, $m=\dim M$, $\frh^k$ be the $k$-dimensional Hausdorff
measure on $M$. Yau conjectured that there exists positive
constant $c$ and $C$ such that
\begin{eqnarray*}
c\sqrt{\la}\leq\frh^{m-1}(N_u)\leq C\sqrt{\la}
\end{eqnarray*}
for the nodal set $N_u$ of any eigenfunction $u$ corresponding to
the eigenvalue $-\la$. For real analytic $M$, this conjecture was
proved by Donnelly and Fefferman in \cite{DF}. In the case of a
surface, lower bounds were obtained in papers \cite{Br} and
\cite{Sa}; in \cite{Sa}, $c=\frac1{11}\mathop{\mathrm{Area}}(M)$.

We consider first the case $M=\bbS^{m}\subset\bbR^{m+1}$,
$m\geq1$. Set
\begin{eqnarray*}
\psi(x)=\Re(x_1+ix_2)^n.
\end{eqnarray*}
Clearly, $\psi\in\cH_n^{m+1}$. Let $\phi$ denote a zonal spherical
harmonic; we omit the index since the geometric quantities that
characterize $N_\phi$ are independent of it. Set
\begin{eqnarray*}
\om_k=\frh^k(\bbS^k)=\frac{2\pi^{\frac{k+1}2}}
{\Ga\left(\frac{k+1}{2}\right)}\,.
\end{eqnarray*}
\begin{theorem}\label{shupp}
For any nonzero real $u\in\cH_n^{m+1}$,
\begin{eqnarray}\label{estno}
\frh^{m-1}(N_u)\leq\frh^{m-1}(N_\psi)=n\om_{m-1}.
\end{eqnarray}
\end{theorem}
The theorem is simply an observation modulo the following fact (a
particular case of Theorem~3.2.48 in \cite{Fe}). A set which can
be realized as the image of a bounded subset of $\bbR^k$ under a
Lipschitz mapping is called {\it $k$-rectifiable} (we consider
only the sets which can be realized as the countable union of
compact sets). Since $u\in\cH_n^{m+1}$ is a polynomial, the set
$N_u$ is $(m-1)$-rectifiable. Let $\mu_m$ denote the invariant
measure on $\OO(m+1)$ with the total mass $1$.
\begin{theorem}[\cite{Fe}]\label{feder}
Let $A,B\subseteq \bbS^{d}$ be compact, $A$ be $k$-rectifiable,
and $B$ be $l$-rectifiable. Set $r=k+l-d$. Suppose $r\geq0$. Then
\begin{eqnarray}\label{intca}
\int_{\OO(d)}\frh^r(A\cap gB)\,d\mu_d(g)=K\frh^k(A)\frh^l(B),
\end{eqnarray} where
$\displaystyle K=\frac{\Ga\left(\frac{k+1}{2}\right)
\Ga\left(\frac{l+1}{2}\right)}{2\Ga\left(\frac{1}{2}\right)^d
\Ga\left(\frac{r+1}{2}\right)}=\frac{\om_r}{\om_k\om_l}$.  \qed
\end{theorem}
If $r=0$, then the left-hand side of (\ref{intca}) is a version of
the Favard measure  for spheres (on $A$ or $B$). Also, note that
(\ref{intca}) can be proved directly in this setting since the
left-hand side, for fixed $A$ (or $B$), is additive on finite
families of disjoint compact sets; thus, it is sufficient to check
its asymptotic behavior on small pieces of submanifolds.
\begin{lemma}
For any real $u\in\cH_n^{m+1}$ and each big circle $\bbS^1$ in
$\bbS^{m}$, if $\bbS^1\cap N_u$ is finite, then
\begin{eqnarray}\label{estci}
\card (\bbS^1\cap N_u)\leq2n.
\end{eqnarray}
\end{lemma}
\begin{proof}
The restriction of $u$ to the linear span of $\bbS^1$, which is
2-dimensional, is a homogeneous polynomial of degree $n$ of two
variables.
\end{proof}
\begin{proof}[Proof of Theorem~\ref{shupp}]
Since $\bbS^1$ intersects  in two points any hyperplane which does
not contain it, for almost all $g\in\OO(m+1)$ we have
\begin{eqnarray*}
\card (g\bbS^1\cap N_u)\leq2n= \card (g\bbS^1\cap N_\psi).
\end{eqnarray*}
Integrating over $\OO(m+1)$ and applying (\ref{intca}) with $k=1$,
$l=m-1$, $A=\bbS^1$, $B=N_u$ and $B=N_\psi$, we get the inequality
in (\ref{estno}). The equality is evident.
\end{proof}
A lower bound can also be obtained in this way. In what follows,
we assume $k=l=1$ and $m=2$; then $K=\frac1{2\pi^2}$, and
(\ref{estno}) read as follows:
\begin{eqnarray}\label{upps2}
\frh^1(N_u)\leq2\pi n.
\end{eqnarray}
The nodal set $N_\phi$ of a zonal spherical harmonic
$\phi=\phi_a\in\cH_n$, where $a\in \bbS^2$, is the union of
parallel circles of Euclidean radii $\sqrt{1-t_k^2}$, where $t_k$
are zeroes of the Legendre polynomial $P_n$. The smallest circle
corresponds to the greatest zero $t_n$. Set $r_n=\sqrt{1-t_n^2}$
and let $C_n$ be a circle in $\bbS^2$ of Euclidean radius $r_n$.
By Proposition~\ref{fromp}, for any $u\in\cH_n$,
\begin{eqnarray}\label{intno}
\card(gC_n\cap N_u)\geq2\quad\mbox{for all}\quad g\in\OO(3).
\end{eqnarray}
Due to (\ref{intca}),
\begin{eqnarray*}
\frh^1(N_u)\geq\frac{2\pi}{r_n}\,.
\end{eqnarray*}
By \cite[Theorem~6.3.4]{Sz}, $t_n=\cos \theta_n$, where
\begin{eqnarray}\label{estth}
0<\theta_n<\frac{j_0}{n+\frac12}
\end{eqnarray}
and $j_0\approx2.4048$ is the least positive zero of Bessel
function $J_0$. This estimate, by \cite[(6.3.15)]{Sz}, is
asymptotically sharp: $\lim_{n\to\infty}n\theta_n=j_0$. Thus,
\begin{eqnarray*}
r_n=\sin\theta_n<\sin\frac{j_0}{n+\frac12}<\frac{j_0}{n+\frac12},
\end{eqnarray*}
and we get
\begin{eqnarray}\label{eslow}
\frh^1(N_u)>\frac{2\pi}{j_0}\left(n+\frac12\right).
\end{eqnarray}
The bound (\ref{eslow}) is not the best one but it is greater than
$\frac1{11}$\,Area$\,(M)\sqrt{\la}$, the bound  of paper
\cite{Sa}:
\begin{eqnarray*}
\frac{4\pi}{11}\sqrt{n(n+1)}<\frac{2\pi}{j_0}\left(n+\frac12\right),
\end{eqnarray*}
since $\frac{4\pi}{11}\approx1.4248$,
$\frac{2\pi}{j_0}\approx2.6127$; according to \cite{Sa},
$\frac1{11}$\,Area$\,(M)\sqrt{\la}\,$ estimates from below the
nodal length for all closed Riemannian surfaces $M$ (for
sufficiently large $\la$ in general and for all $\la$ if the
curvature is nonnegative). The length of the nodal set of a zonal
harmonic could be the sharp lower bound. According to
\cite[(6.21.5)]{Sz},
$\frac{k-\frac12}{n+\frac12}\pi\leq\tau_{n-k}\leq\frac{k}{n+\frac12}\pi$,
where $\cos\tau_k$, $k=0,\dots,n-1$, are the zeroes of $P_n$ in
the order of decreasing (i.e., $\tau_1=\theta_n$). Hence
\begin{eqnarray*}
\frh^1(N_\phi)=2\pi\sum_{k=1}^n\sin\theta_k\approx2\pi
n\int_0^1\sin\pi x\,dx=4n
\end{eqnarray*}
as $n\to\infty$. If this is true, then the upper bound is rather
close to the lower one since their ratio tends to $\frac\pi2$ as
$n\to\infty$.

It is also possible to estimate the {\it inner radius} of
$\bbS^2\setminus N_u$:
\begin{eqnarray*}
\inr(\bbS^2\setminus N_u)=\sup\left\{\,\inf\nolimits_{y\in
N_u}\rho(x,y):\,{x\in\bbS^2}\right\},
\end{eqnarray*}
where $\rho$ is the inner metric in $\bbS^2$:
\begin{eqnarray*}
\rho(x,y)=\arccos\scal{x}{y}.
\end{eqnarray*}
The least upper bound is evident:
\begin{eqnarray*}
\inr(\bbS^2\setminus N_u)\leq\inr(\bbS^2\setminus N_\phi)=\theta_n
\end{eqnarray*}
by (\ref{estth}). Indeed, it is attained for $u=\phi$ and cannot
be greater since the circle $C_n$ intersects any nodal set by
Proposition~\ref{fromp}. Let $C(\theta)$ be the a circle of radii
$\theta$ in the inner metric of $\bbS^2$; then Euclidean radius of
$C(\theta)$ is $r=\sin\theta$. A number $\theta_0>0$ is a lower
bound for the inner radius if and only if the following conditions
hold:
\begin{itemize}
\item[(\romannumeral1)] $\theta_0\leq\theta_n$,
\item[(\romannumeral2)] for each real $u\in\cH_n$, there exists
$g\in\OO(3)$ such that $gC(\theta_0)\cap N_u=\emptyset$.
\end{itemize}
(note that the disc bounded by $C(\theta_0)$ cannot contain a
component of $N_u$ due to (\romannumeral1)). Further, for almost
all $g\in\OO(3)$ the number $\card(gC(\theta_0)\cap N_u)$ is even.
Therefore, we may assume that
\begin{eqnarray*}
\card(gC(\theta_0)\cap N_u)\geq2
\end{eqnarray*}
if $gC(\theta_0)\cap N_u\neq\emptyset$. Set $r_0=\sin\theta_0$. If
(\romannumeral2) is false then
\begin{eqnarray*}
2\leq
\frac1{2\pi^2}\frh^1(C(\theta_0))\frh^1(N_u)=\frac{r_0}{\pi}\,\frh^1(N_u)\leq2r_0n
\end{eqnarray*}
by (\ref{intca}). Thus,  if $r_0<\frac1n$, then $\theta_0$ is a
lower bound for $\inr(\bbS^2\setminus N_u)$. Hence
$\arcsin\frac1{n}$ is a lower bound for $\inr(\bbS^2\setminus
N_u)$. The estimate seems to be non-sharp; perhaps, the least
inner radius has the set $\bbS^2\setminus N_\psi$ (it is equal to
$\frac{\pi}{2n}$).

We summarize the results on $\bbS^2$.
\begin{theorem}
Let $M=\bbS^2$. For any nonzero real $u\in\cH_n$,
\begin{eqnarray}
&\frac{2\pi}{j_0}\left(n+\frac12\right)<\frh^1(N_u) \leq2\pi n,\label{len}\\
&\arcsin\frac1{n}\leq\inr\left(\bbS^2\setminus
N_u\right)\leq\theta_n<\frac{j_0}{n+\frac12}\label{inr}.
\end{eqnarray}
In {\rm(\ref{len})}, the upper bound is attained if $u=\psi$; the
upper bound $\theta_n$ in {\rm(\ref{inr})} is attained for
$u=\phi$.\qed
\end{theorem}

\section{Mean Hausdorff measure of intersections
of the nodal sets}

Let us fix $m\geq2$ and the unit sphere
$\bbS^{m}\subset\bbR^{m+1}$. We shall find the mean value over
$u_1,\dots,u_k$, $k\leq m$, of the Hausdorff measure of sets
\begin{eqnarray*}
N_{u_1}\cap\dots\cap N_{u_k}\subset\bbS^{m}.
\end{eqnarray*}
If $k=m$, then this is the mean number of common zeroes of
$u_1,\dots,u_m$ in $\bbS^{m}$. Set
\begin{eqnarray*}
\bfn=(n_1,\dots,n_k),\\
\de(n)=\dim\cH^{m+1}_n-1,
\end{eqnarray*}
where $n,n_j$ are positive integers. We define the mean value as
follows:
\begin{eqnarray}\label{defme}
M_{\bfn}
=\int_{\bbS^{\de(n_1)}\times\dots\times\bbS^{\de(n_k)}}\!\frh^{m-k}\left(N_{u_1}\cap\dots\cap
N_{u_k}\right)\,d\td\si_{\de(n_1)}(u_1)\dots
d\td\si_{\de(n_k)}(u_k),
\end{eqnarray}
where $\td\si_j$ denotes the invariant measure on $\bbS^j$ with
the total mass 1. Let $\la_n$ be the eigenvalue of $-\De$ in
$\cH^{m+1}_n$; recall that
\begin{eqnarray*}
\la_n=n(n+m-1).
\end{eqnarray*}
\begin{theorem}\label{meanh}
Let $1\leq k\leq m$. Then
\begin{eqnarray}\label{meann}
M_\bfn=\om_{m-k}m^{-\frac{k}{2}}\sqrt{\la_{n_1}\dots \la_{n_k}},
\end{eqnarray}
where $M_\bfn$ is defined by {\rm(\ref{defme})}.
\end{theorem}
If $k=m$, then we get the mean value of
$\card\left(N_{u_1}\cap\dots\cap N_{u_m}\right)$; since $\om_0=2$
and $\frh^0=\card$, it is equal to
\begin{eqnarray*}
2m^{-\frac{m}{2}}\sqrt{\la_{n_1}\dots \la_{n_m}}.
\end{eqnarray*}

There is a natural equivariant immersion
$\iota_n:\,\bbS^{m}\to\bbS^{\de(n)}\subset\cH^{m+1}_n$:
\begin{eqnarray}
\iota_n(a)=\frac{\phi_a}{|\phi_a|}.\label{iotap}
\end{eqnarray}
If $n$ is odd, then $\iota_n$ is one-to-one; for even $n>0$,
$\iota_n$ is a two-sheeted covering, which identifies opposite
points. Clearly, the Riemannian metric in $\iota(\bbS^m)$ is
$\OO(m+1)$-invariant and the stable subgroup of $a$ acts
transitively on spheres in $T_a\bbS^{m}$. Hence, the mapping
$\iota_n$ is locally a metric homothety. Let $s_n$ be its
coefficient. Clearly,
\begin{eqnarray}
s_n=\frac{|d_a\iota_n(v)|}{|v|},\label{defsn}
\end{eqnarray}
where the right-hand side is independent of $a\in\bbS^m$ and $v\in
T_a\bbS^{m}\setminus \{0\}$. For any $l$-rectifiable set
$X\subseteq\bbS^{m}$ such that $X\cap(-X)=\emptyset$, where $l\leq
m$, we have
\begin{eqnarray}\label{iodil}
&\frh^l(\iota_n(X))=s_n^{l}\frh^l(X).
\end{eqnarray}
\begin{lemma}\label{rectx}
Let $u\in\cH_n^{m+1}$ and $X\subseteq\bbS^{m}$ be compact,
symmetric, and $(r+1)$-rectifiable, where $r\leq m-1$. Then
\begin{eqnarray*}
\int_{\bbS^{\de(n)}}\frh^r(N_u\cap X)\,d\si_{\de(n)}(u)=
s_n\frac{\om_r}{\om_{r+1}}\frh^{r+1}(X).
\end{eqnarray*}
\end{lemma}
\begin{proof}
Since both sides are additive on $X$, we may assume
$X\cap(-X)=\emptyset$. We apply Theorem~\ref{feder} to the sphere
$\bbS^{\de(n)}$ and its subsets $A=\bbS^{\de(n)-1}$,
$B=\iota_n(X)$. In the notation of this theorem, $d=\de(n)$,
$k=d-1$, $l=r+1$; $K\om_k=\frac{\om_r}{\om_l}$. Replacing
integration over $\bbS^{d}$ by averaging over $\OO(d+1)$ and using
(\ref{iodil}), we get

\begin{eqnarray*}
\int_{\bbS^{d}}\frh^r(N_u\cap X)\,d\si_{d}(u)=
\frac1{s_n^r}\int_{\bbS^{d}}\frh^r(\iota(N_u\cap
X))\,d\si_{d}(u)\\ =
\frac1{s_n^r}\int_{\bbS^{d}}\frh^r(u^\bot\cap\iota(X))\,d\si_{d}(u)=
\frac1{s_n^r}\int_{\OO(d+1)}\frh^r(g\bbS^{k}\cap\iota(X))\,d\mu_{d}(g)\\=
\frac1{s_n^r} K\frh^{k}\left({\bbS^{k}}\right)\frh^{r+1}(\iota(X))
=\frac{\om_r}{s_n^r\om_{r+1}}\frh^{r+1}(\iota(X))=
s_n\frac{\om_r}{\om_{r+1}}\frh^{r+1}(X).\quad 
\end{eqnarray*}
\end{proof}
\begin{corollary}\label{nodon}
The mean value of $\frh^{m-1}(N_u)$ over $u\in\cH_n^{m+1}$ is
equal to $s_n\om_{m-1}$.
\end{corollary}
\begin{proof}
Set $X=\bbS^{m}$, $r=m-1$.
\end{proof}
\begin{corollary}\label{meana}
Let $M_\bfn$, $m$, and $k$  be as in {\rm(\ref{defme})}. Then
\begin{eqnarray}\label{measn}
M_\bfn=\om_{m-k}\prod_{j=1}^k s_{n_j}.
\end{eqnarray}
\end{corollary}
\begin{proof}
Set $X=N_{u_1}\cap\dots\cap N_{u_{k-1}}$. By Lemma~\ref{rectx},
\begin{eqnarray*}
M_\bfn=s_{n_k}\frac{\om_{m-k}}{\om_{m-k+1}}M_{\bfn'},
\end{eqnarray*}
where $\bfn'=(n_1,\dots,n_{k-1})$. Applying this procedure
repeatedly and using Corollary~\ref{nodon} in the final step, we
get (\ref{measn}).
\end{proof}
It remains to find $s_n$. Set
\begin{eqnarray*}
d=\dim\OO(m+1).
\end{eqnarray*}
Since the stable subgroup $\OO(m)$ of the point $a=(0,\dots,0,1)$
acts transitively on spheres in $T_a\bbS^{m}$, the invariant
Riemannian metric in $\bbS^{m}$ can be lifted up to a bi-invariant
metric on $\OO(m+1)$ in such a way that the canonical projection
$\OO(m+1)\to\bbS^{m}$ is a metric submersion. Let
$\xi_1,\dots,\xi_{m},\dots,\xi_d$ be an orthonormal linear base in
the Lie algebra $\so(m+1)$. Realizing $\so(m+1)$ by the left
invariant
vector fields on $\OO(m+1)$, 
we get the invariant Laplace--Beltrami operator on $\OO(m+1)$:
\begin{eqnarray*}
\td\De=\xi_1^2+\dots+\xi_d^2.
\end{eqnarray*}
The sum is independent of the choice of the base since it is left
invariant and this property holds at the identity element $e$.
Thus, we may assume that
\begin{eqnarray}\label{insta}
\xi_{m+1},\dots,\xi_d\in\so(m).
\end{eqnarray}
For $f\in C^2(\bbS^{m})$, set $\td f(g)=f(ga)$. Then $\scal{\De
f}{\phi_a}=\td\De\td f(e)$. Since $\iota$ is equivariant, we have
\begin{eqnarray}\label{difio}
d_a\iota(\xi a)=\frac1{|\phi_a|}\xi\phi_a
\end{eqnarray}
for all $\xi\in\so(m+1)$. It follows from (\ref{insta}) that
$\xi_1 a,\dots,\xi_{m}a$ is a base for $T_a\bbS^{m}$ and
$\xi_1\phi_a,\dots\xi_{m}\phi_a$ is a base for
$T_{\phi_a}\iota(\bbS^{m})$. Moreover,
\begin{eqnarray*}
&|\xi_k a|=1,\quad k=1,\dots,m,\\
&\xi_k a=0,\quad k=m+1,\dots,d,
\end{eqnarray*}
where the first equality holds since the projection
$\OO(m+1)\to\bbS^{m}$ is a metric submersion.  Due to these
equalities,  (\ref{iotap}), (\ref{defsn}), and (\ref{difio}), we
get
\begin{eqnarray*}
ms_n^2=s_n^2\sum_{k=1}^{d}|\xi_k
a|^2=\sum_{k=1}^{d}|d_a\iota(\xi_ka)|^2=
\frac1{|\phi_a|^2}\sum_{k=1}^{d}|\xi_k\phi_a|^2\\
=-\frac1{|\phi_a|^2}\sum_{k=1}^{d}\scal{\xi_k^2\phi_a}{\phi_a}
=-\frac1{|\phi_a|^2}\scal{\De\phi_a}{\phi_a} =\la_n.
\end{eqnarray*}
\begin{proof}[Proof of Theorem~\ref{meanh}]
Due to the calculation above,
\begin{eqnarray*}
s_n=\sqrt{\frac{\la_n}{m}}.
\end{eqnarray*}
Thus, Corollary~\ref{meana} implies (\ref{meann}).
\end{proof}
In the case $n_1=\dots=n_k=n$, there is another natural
explanation of the equalities (\ref{meann}), (\ref{measn}):
\begin{eqnarray*}
M_\bfn=\om_{m-k}\left(\frac{\la_n}{m}\right)^{\frac{k}2}=\om_{m-k}s_n^k.
\end{eqnarray*}
The mean value can be defined as the average over the action of
the group $\OO(m+1)$ on the set of subspaces of codimension $k$ in
$\cH_n^{m+1}$, which can be realized as
$\cN_{u_1}\cap\dots\cap\cN_{u_k}=u_1^\bot\cap\dots\cap u_k^\bot$:
\begin{eqnarray*}
M_\bfn&=
\int_{\OO(m+1)}\frh^{m-k}(\iota_n^{-1}(g\bbS^{\de(n)-k}\cap\iota_n(\bbS^m)))\,d\mu_m(g)\\
&=s_n^{k-m}
\int_{\OO(m+1)}\frh^{m-k}(g\bbS^{\de(n)-k}\cap\iota_n(\bbS^m))\,d\mu_m(g)\\
&\phantom{xxxxxxxxxxxxxxxxxxxx}
=s_n^{k-m}\frac{\om_{m-k}}{\om_m}\frh^{m}(\iota(\bbS^m))&=
\om_{m-k}s_n^k.
\end{eqnarray*}
The method of calculation of the mean Hausdorff measure easily can
be extended to families of invariant (may be, reducible) finite
dimensional function spaces on a homogeneous space whose isotropy
group acts transitively on spheres in the tangent space.

\medskip

{\bf Acknowledgements.} I am grateful to D.\,Jakobson for useful
references and comments and to L.\,Polterovich for his
question/conjecture on ``the Bezout theorem in the mean''.

\vskip1cm
V.M. Gichev\\
Omsk Branch of Sobolev Institute of Mathematics\\
Pevtsova, 13, Omsk, 644099, Russia\\
\medskip gichev@ofim.oscsbras.ru
\end{document}